\documentclass[11pt]{article}

\usepackage{epsfig}
\usepackage[latin1]{inputenc}
\usepackage{graphicx,psfrag} 
\usepackage[tight]{subfigure}

\usepackage{amsfonts,amsthm,bbm,amssymb,epsf,amsmath,%
latexsym,amscd,amsfonts,enumerate,amsthm,supertabular,color,url}

\newtheorem{theorem}{Theorem}
\newtheorem{lemma}{Lemma}
\newtheorem{corollary}{Corollary}
\newtheorem{proposition}{Proposition}

\newtheorem{remark}{Remark}

\newcommand{\RR}{\mathbb{R}}

\newcommand{\NN}{\mathbb{N}}

\def\eps{\varepsilon}

\def\z{{\bf z}}
\def\w{{\bf w}}

\def\x{{\bf x}}

\def\y{{\bf y}}

\def\X{{\bf X}}

\def\cL{\mathcal{L}}
\def\cO{\mathcal{O}}

\begin{document}

\title{Error bounds for numerical differentiation\\ using kernels of finite smoothness}

\author{Oleg Davydov\thanks{Department of Mathematics, Justus Liebig University of Giessen, Arndtstrasse 2, 
35392 Giessen, Germany, \tt{oleg.davydov@math.uni-giessen.de}}
}

\date{November 15, 2025}

\maketitle

\begin{abstract}
We provide improved error bounds for kernel-based numerical differentiation in terms of growth functions when kernels are
of a finite smoothness, such as polyharmonic splines, thin plate splines or Wendland kernels. In contrast to existing
literature, the new estimates take into account the H\"older class smoothness of kernel's derivatives, which helps to improve
the order of the estimate. In addition, the new estimates apply to certain deficient point sets, relaxing a standard
assumption that an approximation with conditionally positive definite kernels must rely on determining sets for polynomials. 
\end{abstract}

\section{Introduction}\label{intro}

Let 
$$
K:\Omega \times \Omega  \to \RR
$$
be a symmetric conditionally positive definite (cpd) kernel on $\Omega\subset\RR^d$ of order $s\ge0$, see for example \cite{Wendland}, and $D$ a 
linear differential operator of order $k\ge0$ in $d$ real variables 
\begin{equation}\label{Dop}  
Df=\sum_{|\alpha|\le k}a_\alpha\partial^\alpha\! f,
\quad
\partial^\alpha:=\frac{\partial^{|\alpha|}}{\partial \x^{\alpha}}=
\frac{\partial^{|\alpha|}}{\partial x_1^{\alpha_1}\cdots\partial x_d^{\alpha_d}},
\quad|\alpha|=\alpha_1+\cdots+\alpha_d,
\end{equation}
with variable coefficients $a_\alpha$. In particular, for $k=0$ and $a_{(0,\ldots,0)}=1$, $D$ is the identity operator $Df=f$.

For the operator $D$ or any other operator or functional $A$ applicable to real-valued function of $d$ variables, 
we will use the notation $A'K$ and $A''K$ to denote the result of applying $A$ to the first,
respectively, second $d$-dimensional argument of the kernel $K$. If one operator is applied to the first and another to the
second argument of $K$, then we write for example $A'B''K$ when first $B$ was applied to the second argument, 
and then $A$ to the first. For the partial derivative operators $\partial^\alpha$ that are applied 
repeatedly to the first and second arguments of $K$, we use a simplified notation 
$$\partial^{\alpha,\beta}K(\x,\y)
:=\frac{\partial^{|\alpha|}}{\partial \x^{\alpha}}\frac{\partial^{|\beta|}}{\partial \y^{\beta}}K(\x,\y)
=\frac{\partial^{|\alpha|+|\beta|}}{\partial x_1^{\alpha_1}
\cdots\partial x_d^{\alpha_d}\partial y_1^{\alpha_1}\cdots\partial y_d^{\alpha_d}}K(\x,\y)$$
whenever the ordering of partial differentiation does not influence the result.

Given a point $\z\in\RR^d$ and a finite set $\X=\{\x_1,\ldots,\x_N\}\subset\Omega$, a \emph{kernel-based 
numerical differentiation formula}
\begin{equation}\label{ndif} 
D f(\z)\approx\sum_{j=1}^{N} w^*_j\, f(\x_j)
\end{equation}
with weight vector $\w^*=[w^*_j]_{j=1}^N$ is obtained by solving the linear system
\begin{alignat}{4}
\label{wsys}
     &\sum_{j=1}^{N} w_j K( \x_i, \x_j )&&+\sum_{j=1}^M v_jp_j(x_i)
&&= D'K(\z,\x_i),\qquad &&1\leq i\leq N,\\
\label{wsysp}
     &\sum_{j=1}^{N} w_j p_i(\x_j)&&%
&&= Dp_i(\z),\qquad &&1\leq i\leq M,
\end{alignat}
with unknowns $w_1\ldots,w_N$ and $v_1\ldots,v_M$, where 
$p_1,\ldots,p_M$ is a basis for the space $\Pi^d_s$ of $d$-variate polynomials of total degree at most $s-1$, such that
$M={d+s-1\choose d}$. %
In the case $s=0$ the kernel is strictly positive
definite and we set $\Pi^d_s=\{0\}$ and $M=0$.

It has been shown in \cite{D21} that the weights $w^*_j=w_j$, $j=1,\ldots,N$, are uniquely determined by \eqref{wsys}--\eqref{wsysp} as soon as
\eqref{wsysp} is consistent, that is, there is at least one vector $\w\in\RR^N$ satisfying \eqref{wsysp}. This extends the
more restrictive condition that the set  $\X$ is a determining set for $\Pi^d_s$, in which case the right hand side of 
formula \eqref{ndif} expresses the values at $\z$ of $D$ applied to the kernel-based  interpolant of $f$, see \cite{DS16}.
Usually, \eqref{wsysp} has many solutions, but only one of them satisfies \eqref{wsys}. The coefficients $v_j$ may not be
uniquely determined, but they are not needed for numerical differentiation.

Accurate numerical differentiation formulas of type \eqref{ndif} are needed in the meshless finite difference methods 
for partial differential equations, such as RBF-FD, see for example \cite{FFprimer15}. Asymptotic behavior of the error of
\eqref{ndif} when $s$ is fixed and $N\to\infty$ is usually estimated in terms of so-called \emph{fill distance} 
of the set $\X$ in $\Omega$ \cite{Wendland}. Applications to meshless finite difference methods however
requires estimates applicable to $N$ not significantly exceeding $M$ \cite{BFFB17}.
Such estimates using a \emph{growth function} instead of the fill distance have been proposed in \cite{DS16}.
However, as we will see in Section~\ref{trinv} the results of \cite{DS16} provide suboptimal error estimates in the case when
the kernels are of a finite smoothness. 
In this paper we improve the error estimates of \cite{DS16} by taking into account the fact that the
highest order continuous partial derivatives  of such kernels belong to certain H\"older smoothness classes. Moreover, we show
that the error bounds remain valid under the less restrictive assumption of solvability of \eqref{wsysp},
which allows to apply them to certain deficient sets as demonstrated numerically in \cite{D21}.

The paper is organized as follows. Section~\ref{main} is devoted to 
the main general results that are subsequently applied in Section~\ref{trinv} to the particularly important translation-invariant 
kernels of finite smoothness, with full details provided for polyharmonic splines, thin plate splines and Wendland kernels.

\section{Main results}\label{main}

We start by summarizing some basic facts about kernels and numerical differentiation formulas that will be needed in what
follows.

Theorems 1 and 2 in \cite{D21} imply the following statement.

\begin{proposition}\label{derivlemma} 
Let $D$ be a linear differential operator of order $k$ and let $K$ be a cpd kernel of
order $s$. The kernel--based numerical differentiation formula
(\ref{ndif}) is well defined 
as long as $D' K(\z,\x_i)$ exists for all $i=1\ldots,N$, and the linear system \eqref{wsysp} is consistent. 
The formula (\ref{ndif}) is exact for all functions $f$ of the form
\begin{equation}
\label{radtype}
     \sum_{i=1}^{N} a_i K( \cdot, \x_i )+\sum_{i=1}^M b_ip_i,\quad a_i,b_i\in\RR,
\end{equation}
satisfying 
\begin{equation}
\label{radcond}
     \sum_{j=1}^{N} a_jp_i(\x_j)=0,\quad 1\le i\le M.
\end{equation}
In particular, (\ref{ndif}) is exact for all polynomials in  $\Pi_{s}^d$.
\end{proposition}
\noindent
Note that exactness for polynomials is obvious from \eqref{wsysp}.

\medskip

Let ${\cal F}_{K}$ denote the {\em native space} of the kernel $K$ \cite{Schaback99,Wendland}. It is a semi-Hilbert space of
real valued functions on $\Omega$ arising as completion of the space of 
functions of the form (\ref{radtype}) satisfying \eqref{radcond}. 
Moreover, ${\cal F}_{K}$ is the direct sum of a Hilbert space and the polynomial space $\Pi^d_{s}\subset \ker {\cal F}_{K}$.
We denote by $(f,g)_K$ and $\|f\|_K=(f,f)_K^{1/2}$ the inner product and the (semi-)norm of ${\cal F}_{K}$.
We also denote by $\|\lambda\|_K$ the standard norm of any bounded linear functional $\lambda\in {\cal F}_{K}^*$.

In the case $s=0$ the strictly positive definite kernel $K$ is the reproducing kernel of ${\cal F}_{K}$. Otherwise, $K$
possesses a weaker reproduction property
\begin{equation}\label{reprL}
\lambda f=(f,\lambda''K)_K\quad\text{for all }\lambda\in \cL(\Pi^d_{s},\Omega)
\end{equation}
where $\cL(\Pi^d_{s},\Omega)$ consists of all  linear functionals $\lambda\in {\cal F}_{K}^*$ of the form
$\lambda f=\sum_{j=1}^{N} a_jf(\x_j)$ for arbitrary $\X\subset\Omega$ and coefficients $a_j$ such that \eqref{radcond} holds,
that is $\lambda p=0$ for all $p\in \Pi^d_{s}$. This property follows for example from \cite[Theorem~6.1]{Schaback99} 
since for the auxiliary kernel $\Psi$ used there the difference 
$\lambda''K-\lambda''\Psi$ belongs to $\Pi^d_{s}$ and hence it is orthogonal to any element $f$ of ${\cal F}_{K}$.

Using \eqref{reprL} we can argue as in \cite[Section~2]{DS16} to prove the following result.

\begin{proposition}\label{diffop} 
Let $D$ be a linear differential operator of order $k$, and $K$ a cpd kernel of
order $s$. The error functional 
\begin{equation}\label{diffunct}
\eps_{\w}: f\mapsto Df(\z) -\displaystyle{\sum_{j=1}^N w_jf(\x_j)   } 
\end{equation}
of any numerical differentiation formula  satisfying $\eps_{\w}p=0$ for all $p\in \Pi^d_{s}$ is continuous on  
$ {\cal F}_{K}$ if $\partial^{\alpha,\beta} K(\z,\z)$
exists for all $|\alpha|,|\beta|\le k$. Moreover,
\begin{equation}\label{munorm}
\|\eps_{\w}\|_K^2= \eps_{\w}'\eps_{\w}''K. 
\end{equation}
\end{proposition}
\noindent
Note that the only improvement in comparison to \cite[Lemma~4]{DS16} is that we do not require from $\X$ to be a determining
set for $\Pi^d_{s}$.

\medskip

Similarly, without changing its proof we obtain a slight improvement of \cite[Lemma~6]{DS16}:

\begin{proposition}
\label{powerK}
Let $D$ be a linear differential operator of order $k$, and $K$ a cpd kernel of order $s$ on $\Omega\subset\RR^d$. 
For $\z\in\Omega$ and $\X=\{\x_1,\ldots,\x_N\}\subset\Omega$, assume that 
\begin{itemize}
\item[(a)]
$\partial^{\alpha,\beta} K(\z,\z)$ exists for all $|\alpha|,|\beta|\le k$,
\item[(b)]
$D' K(\z,\x_i)$ exists for all $i=1\ldots,N$, and
\item[(c)]
the linear system 
\eqref{wsysp} is consistent. 
\end{itemize}
Then the kernel-based
numerical differentiation formula \eqref{ndif} satisfies
\begin{equation}\label{PestK}
\Big|Df(\z) -\sum_{j=1}^N w^*_jf(\x_j)\Big|\le P_{K,D,\X,\z}\,\|f\|_{K}\quad\text{for all }f\in {\cal F}_{K},
\end{equation}
where $P_{K,D,\X,\z}$ is the \emph{power function}
defined by
\begin{equation}\label{est2p}
P_{K,D,\X,\z}^2=\min\Big\{Q_{K,D,\X,\z}(\w):\w\in\RR^{N},\;
\eps_{\w}p=0\;\hbox{for all }p\in\Pi^d_{s}\Big\},
\end{equation}
with 
\begin{equation} \label{QlX}
Q_{K,D,\X,\z}(\w):= \eps_{\w}'\eps_{\w}''K.
\end{equation}
\end{proposition}

\medskip

\noindent
Note that the formula \eqref{ndif} is well defined by 
Proposition~\ref{derivlemma}.

\medskip

In view of Proposition~\ref{diffop}, $Q_{K,D,\X,\z}(\w)$ is the square of the worst case error of numerical differentiation using weight
vectors $\w$ satisfying the polynomial exactness condition $\eps_{\w}p=0$ for all $p\in\Pi^d_s$, which we will briefly refer 
to as $\eps_{\w}(\Pi^d_s)=0$, for functions in the unit ball of $\mathcal{F}_K$,
\begin{equation}\label{optest}
Q_{K,D,\X,\z}(\w)=\sup_{\substack{f\in \mathcal{F}_K\\ \|f\|_K\le 1}}\Big|Df(\z)-\sum_{j=1}^Nw_jf(x_j)\Big|^2
\quad\text{ if $\eps_{\w}(\Pi^d_s)=0$}.
\end{equation}
Moreover, Proposition~\eqref{powerK} shows that the kernel-based numerical differentiation formula \eqref{ndif}
provides the \emph{optimal recovery} 
of $Df(\z)$ for $f$ in the native space $\mathcal{F}_K$ among all numerical differentiation formulas
exact for polynomials in $\Pi^d_{s}$, and the power function gives the error of the optimal recovery,
\begin{equation}\label{optrec}
P_{K,D,\X,\z}=\inf_{\substack{w\in\RR^N\\ \eps_{\w}(\Pi^d_s)=0}}\;
\sup_{\substack{f\in \mathcal{F}_K\\ \|f\|_K\le 1}}\Big|Df(\z)-\sum_{j=1}^Nw_jf(x_j)\Big|.
\end{equation}

\begin{remark} We can easily compute $Q_{K,D,\X,\z}(\w)$ for any weight vector $\w$,
\begin{equation}\label{Qform}
Q_{K,D,\X,\z}(\w)=D'D''K(\z,\z)-2\sum_{j=1}^N\! w_jD'K(\z,\x_j)+\!\sum_{i,j=1}^N\!w_iw_jK(\x_i,\x_j),
\end{equation}
where we used the identity $D''K(\x_j,\z)=D'K(\z,\x_j)$ that holds in view of the symmetry of $K$.
For $\w=\w^*$ the formula simplifies further to 
\begin{equation}\label{Pform}
P_{K,D,\X,\z}^2=Q_{K,D,\X,\z}(\w^*)=D'D''K(\z,\z)-\sum_{j=1}^N w^*_jD'K(\z,\x_j)
\end{equation}
since 
$$
D'K(\z,\x_j)=\sum_{i=1}^{N} w^*_i K( \x_j, \x_i )+\sum_{i=1}^M v_ip_i(x_j)$$
and therefore
$$
\sum_{j=1}^N w^*_jD'K(\z,\x_j)=\sum_{i,j=1}^N w^*_iw^*_jK(\x_i,\x_j)$$
by \eqref{wsys}--\eqref{wsysp}.
\end{remark}

We now extend the upper bound for $Q_{K,D,\X,\z}(\w)$ given in \cite[Lemma 7]{DS16} to kernels with weaker smoothness
assumptions and all sets $\X$ for which \eqref{wsysp} is consistent.

Let
$$
S_{\z,\X}:=\bigcup_{i=1}^N[\z,\x_i].$$
For any function $U:S_{\z,\X}\times S_{\z,\X}\to \RR$, we set
$$
\big|U\big|_{\z,\X,\gamma}
:=\sup_{\substack{\x,\y\in S_{\z,\X}\\ \x\ne\z,\,\y\ne\z}}
\frac{\big|U(\x,\y)-U(\x,\z)-U(\z,\y)+U(\z,\z)\big|}{\|\x-\z\|_2^\gamma\,\|\y-\z\|_2^\gamma},
\quad 0<\gamma\le 1,$$
and for $r\in\NN$, $0<\gamma\le 1$, we set
$$
|U|_{\z,\X,r+\gamma}:=\frac1{(\gamma+1)^2\cdots (\gamma+r)^2}
\Big(\sum_{|\alpha|=r}\;\sum_{|\beta|=r}{r\choose \alpha}{r\choose \beta}
\big|\partial^{\alpha,\beta} U\big|_{\z,\X,\gamma}^2\Big)^{1/2},
$$
where ${r\choose \alpha}:=\frac{r!}{\alpha!}$, $\alpha!:=\alpha_1!\cdots\alpha_d!$, whenever $U$ is defined on
$\Omega\times\Omega$ for a domain $\Omega$ containing $S_{\z,\X}$,
and derivatives $\partial^{\alpha,\beta} U(\x,\y)$, with $|\alpha|=r$ and $|\beta|=r$ exist for all $\x,\y\in S_{\z,\X}$.

\begin{proposition}
\label{fpestbt}
Let $D$ be a linear differential operator of order $k$, and $K$ a cpd kernel of order $s$ on $\Omega\subset\RR^d$. 
Given $\z\in\Omega$, $\X=\{\x_1,\ldots,\x_N\}$, with $S_{\z,\X}\subset\Omega$, and $\w\in\RR^N$, 
assume that for some integer $q\ge \max\{s,k+1\}$ and $m$ such that  $k\le m\le q-1$, 
\begin{itemize}
\item[(a)] $K$ possesses continuous  partial derivatives  $\partial^{\alpha,\beta} K(\x,\y)$  for all
$|\alpha|,|\beta|\le m$ and all  $\x,\y\in S_{\z,\X}$,  
\item[(b)] $|K|_{\z,\X,m+\gamma}<\infty$, for some $\gamma\in(0,1]$, and 
\item[(c)] $\w$ satisfies
$D p(\z) =\sum_{j=1}^{N} w_jp( \x_j )$ for all $p\in\Pi^d_q$.
\end{itemize}
Then
\begin{equation}\label{Qestbt2}
Q_{K,D,\X,\z}(\w)\le \Big(\sum_{i=1}^N|w_i|\,\|\x_i-\z\|_2^{m+\gamma}\Big)^2|K|_{\z,\X,m+\gamma}.
\end{equation}
\end{proposition}

\begin{proof} Observe that $Q_{K,D,\X,\z}(\w)$, given by \eqref{QlX}, is well defined because all assumptions of
 Proposition~\ref{powerK} are satisfied. In particular, \eqref{wsysp} is consistent thanks to (c) because $q\ge s$.

We adapt in part the proofs of \cite[Lemma 7]{DS16} and \cite[Lemma~2]{DS18}.
Denote by $T_{q,\z}f$ the  Taylor polynomial of degree $m$ of a function $f:\Omega\to\RR$ centered at $\z$,
$$%
T_{m,\z}f(\x)=\sum_{|\alpha|\le m}\frac{(\x-\z)^\alpha}{\alpha!}\partial^\alpha f(\z).
$$%
As in \cite[Lemma~2]{DS18}, we see that for any point $\x\in S_{\z,\X}$ the remainder
$$
R_{m,\z}f(\x)=f(\x)-T_{m,\z}f(\x)$$
can be represented as
$$%
R_{m,\z}f(\x)=m\sum_{|\alpha|=m}\frac{(\x-\z)^\alpha}{\alpha!}
\int_0^1(1-t)^{m-1}[\partial^\alpha f(\z+t(\x-\z))-\partial^\alpha f(\z)]\,dt,
$$%
as long as $f$ is $m$ times continuously differentiable everywhere in $S_{\z,\X}$.
When an expression $U(\x,\y)$ depends on two $d$-dimensional variables $\x$ and $\y$, we write 
$$
T^\x_{m,\z}U(\x,\y)=\sum_{|\alpha|\le m}\frac{(\x-\z)^\alpha}{\alpha!}
\partial^{\alpha,0}U(\z,\y),$$
respectively,
$$
T^\y_{m,\z}U(\x,\y)=\sum_{|\alpha|\le m}\frac{(\y-\z)^\alpha}{\alpha!}
\partial^{0,\alpha}U(\x,\z),$$
for the Taylor polynomial in $\x$, respectively, $\y$, and use a similar notation 
$R^\x_{m,\z}U(\x,\y)=U(\x,\y)-T^\x_{m,\z}U(\x,\y)$ and $R^\y_{m,\z}U(\x,\y)=U(\x,\y)-T^\y_{m,\z}U(\x,\y)$
for the respective remainders.

Following the proof of \cite[Lemma 7]{DS16}, we consider the Boolean sum $B_{q,\z}(\x,\y)$ of Taylor polynomials 
of $K(\x,\y)$,
\begin{align*}
B_{m,\z}(\x,\y)&:=T^\x_{m,\z}K(\x,\y)+T^\y_{m,\z}K(\x,\y)-T^\x_{m,\z}T^\y_{m,\z}K(\x,\y)\\
&=\sum_{|\alpha|\le m}\frac{(\x-\z)^\alpha}{\alpha!}\partial^{\alpha,0}K(\z,\y)
+\sum_{|\alpha|\le m}\frac{(\y-\z)^\alpha}{\alpha!}\partial^{0,\alpha}K(\x,\z)\\
&\quad - \sum_{|\alpha|,|\beta|\le m}\frac{(\x-\z)^\alpha(\y-\z)^\beta}{\alpha!\beta!}
\partial^{\alpha,\beta}K(\z,\z),
\end{align*}
and conclude that
$$
K(\x,\y)-B_{m,\z}(\x,\y)=R(\x,\y):=R^\x_{m,\z}R^\y_{m,\z}K(\x,\y)\quad\text{for all }\x,\y\in\Omega_\z,$$
and, thanks to Assumption (c), 
\begin{equation}\label{Qlxw}
Q_{K,D,\X,\z}(\w)=\sum_{i,j=1}^N w_iw_jR(\x_i,\x_j).
\end{equation}

The above remainder formula for the Taylor polynomial implies 
that
$$
R(\x_i,\x_j)=m\!\!\sum_{|\alpha|=m}\!\frac{(\x_i-\z)^\alpha}{\alpha!}
\int_0^1(1-t)^{m-1}\big[\partial^{\alpha,0} \tilde R(\x^t_i,\x_j)
-\partial^{\alpha,0} \tilde R(\z,\x_j)\big]dt,$$
where $\x^t_i:=\z+t(\x_i-\z)$, and 
$$
\tilde R(\x,\y):=R^\y_{m,\z}K(\x,\y)=K(\x,\y)-T^\y_{m,\z}K(\x,\y)$$
satisfies
$$
\partial^{\alpha,0}\tilde R(\x,\x_j)=m\!\!\!\!\sum_{|\beta|=m}\!\!\!\frac{(\x_j-\z)^\beta}{\beta!}
\!\!\!\int_0^1(1-s)^{m-1}\big[\partial^{\alpha,\beta} K(\x,\x^s_j)
-\partial^{\alpha,\beta} K(\x,\z)\big]ds.$$
Since
\begin{align*}
\big|\partial^{\alpha,\beta} K(\x^t_i,\x^s_j)
&-\partial^{\alpha,\beta} K(\x^t_i,\z)-\partial^{\alpha,\beta} K(\z,\x^s_j)
+\partial^{\alpha,\beta} K(\z,\z)\big|\\
&\le\|\x^t_i-\z\|_2^\gamma\,\|\x^s_j-\z\|_2^\gamma\,\big|\partial^{\alpha,\beta} K\big|_{\z,\X,\gamma}\\
&=(ts)^\gamma\|\x_i-\z\|_2^\gamma\,\|\x_j-\z\|_2^\gamma\,\big|\partial^{\alpha,\beta} K\big|_{\z,\X,\gamma},
\end{align*}
and
$$
\int_0^1(1-t)^{m-1}t^\gamma dt=\frac{(m-1)!}{(\gamma+1)\cdots (\gamma+m)},$$
we obtain
\begin{align*}
\big|R(\x_i,\x_j)\big|\le\;&\Big(\frac{m!}{(\gamma+1)\cdots (\gamma+m)}\Big)^2\|\x_i-\z\|_2^\gamma\,\|\x_j-\z\|_2^\gamma\\
&\sum_{|\alpha|=m}\!\frac{|(\x_i-\z)^\alpha|}{\alpha!}\sum_{|\beta|=m}\!\!\!\frac{|(\x_j-\z)^\beta|}{\beta!}
\,\big|\partial^{\alpha,\beta} K\big|_{\z,\X,\gamma}.
\end{align*}
By the multinomial theorem,
\begin{equation*}%
\sum_{|\alpha|=m}\frac{(\x-\z)^{2\alpha}}{\alpha!}=\frac{1}{m!}\|\x-\z\|_2^{2m}.
\end{equation*}
Hence, it follows by the Cauchy-Schwarz inequality, that
\begin{align*}
\sum_{|\alpha|=m}\!\frac{|(\x_i-\z)^\alpha|}{\alpha!}&\sum_{|\beta|=m}\!\!\!\frac{|(\x_j-\z)^\beta|}{\beta!}
\,\big|\partial^{\alpha,\beta} K\big|_{\z,\X,\gamma}\\
&\!\!\!\le \frac{\|\x_i-\z\|_2^m\|\x_j-\z\|_2^m}{m!}
\Big(\sum_{|\alpha|=m}\;\sum_{|\beta|=m}
\frac{\big|\partial^{\alpha,\beta} K\big|_{\z,\X,\gamma}^2}{\alpha!\beta!}\Big)^{1/2},
\end{align*}
which implies 
$$
\big|R(\x_i,\x_j)\big|\le \|\x_i-\z\|_2^{m+\gamma}\|\x_j-\z\|_2^{m+\gamma}|K|_{\z,\X,m+\gamma}.$$
In view of this, \eqref{Qestbt2} follows from \eqref{Qlxw}.
\end{proof}

\medskip

Recall from \cite{DS18} that the \emph{growth function}
$$
\rho_{q,D}(\z,\X,\mu):=\sup\{D p(\z):p\in \Pi_q^d,\;
|p(\x_j)|\leq \|\x_j-\z\|_2^{\mu},\;j=1,\ldots,N\},$$
where $\mu\ge 0$, coincides with  
\begin{align*}
\inf\Big\{\sum_{j=1}^{N} |w_j|\|\x_j-\z\|_2^{\mu}:\w\in\RR^N,\,D p(\z)=\sum_{j=1}^Nw_jp(\x_j)
\text{ for all }p\in\Pi_q^d\Big\}.
\end{align*}

Therefore, Propositions~\ref{powerK} and \ref{fpestbt} imply the following estimate of the error of the kernel-based
numerical differentiation.

\begin{theorem}\label{pestbt}
Let $D$ be a linear differential operator of order $k$, and $K$ a cpd kernel of order $s$ on $\Omega\subset\RR^d$. 
Given $\z\in\Omega$ and $\X=\{\x_1,\ldots,\x_N\}$, with $S_{\z,\X}\subset\Omega$,  assume that  the linear system 
\eqref{wsysp} is consistent, and for some $r\ge k$ and $\gamma\in(0,1]$, the kernel $K$ possesses continuous  partial
derivatives   $\partial^{\alpha,\beta} K(\x,\y)$  for all $|\alpha|,|\beta|\le r$ and all  $\x,\y\in S_{\z,\X}$, and 
$|K|_{\z,\X,r+\gamma}<\infty$. Then for all $f\in {\cal F}_{K}$ the kernel-based numerical differentiation formula
\eqref{ndif} satisfies the error bound
\begin{equation}\label{gpest}
\Big|Df(\z) -\sum_{j=1}^N w^*_jf(\x_j)\Big|\le \rho_{q,D}(\z,\X,r+\gamma)
\,|K|_{\z,\X,r+\gamma}^{1/2}\,\|f\|_{K},
\end{equation}
where $q=\max\{s, r+1\}$.
\end{theorem}

Since $\rho_{q,D}(\z,\X,\mu)=\sigma\, h_{\z,\X}^{\mu-k}$, with a scale-independent factor $\sigma$,
where 
$$%
h_{\z,\X}:=\displaystyle{  \max_{\x\in\X}\|\z-\x\|_2}.
$$%
see \cite[Section 4.2]{DS18}, we infer that \eqref{gpest} gives an $\cO(h_{\z,\X}^{r+\gamma-k})$ error bound for constellations $\z,\X$ in a
good geometric position. Note that $\rho_{q,D}(\z,\X,\mu)$ can be efficiently estimated  by least squares methods, 
see \cite[Sections 5 and 6]{DS18}.

By applying Proposition~\ref{fpestbt} with $k\le m\le r-1$ instead of $m=r$, and choosing $q=\max\{s,m+1\}$, 
and taking into account that $|K|_{\z,\X,q}<\infty$ if $\partial^{\alpha,\beta} K(\x,\y)$ are continuous 
for all $|\alpha|,|\beta|\le q$  and all  $\x,\y\in S_{\z,\X}$, we obtain the following statement.

\begin{corollary}\label{estold}
Under the hypotheses of Theorem~\ref{pestbt}, for all $f\in {\cal F}_{K}$ and all $q$ satisfying $\max\{s,k+1\}\le q\le r$, 
\begin{equation}\label{gpestold}
\Big|Df(\z) -\sum_{j=1}^N w^*_jf(\x_j)\Big|\le \rho_{q,D}(\z,\X,q)
\,|K|_{\z,\X,q}^{1/2}\,\|f\|_{K}.
\end{equation}
\end{corollary}
Note that Corollary~\ref{estold} makes the same smoothness assumptions about $K$ as \cite[Theorem 9]{DS16}, and only relaxes
the condition on $\X$ that does not anymore have to be a determining set for $\Pi^d_s$, and uses a slightly different 
seminorm $|K|_{\z,\X,q}$ of $K$. In the same time, the estimate \eqref{gpest} of Theorem~\ref{pestbt} provides a higher
approximation order of numerical differentiation for certain kernels $K$, as we will demonstrate in the next section.

\section{Applications to translation-invariant kernels}\label{trinv}

We now consider the \emph{translation-invariant kernels} $K(\x,\y)=\Phi(\x-\y)$, for a function $\Phi:\RR^d\to\RR$, which in
particular covers the \emph{radial basis functions} $\Phi(\x)=\varphi(\|\x\|_2)$, where  $\varphi:[0,\infty)\to\RR$.

Denote by $C^{r,\gamma}(\Omega)$, $r=0,1,\ldots$, $0<\gamma\le 1$, $\Omega\subset\RR^d$, 
the \emph{H\"older space} that consists of all $r$ times
continuously differentiable functions $f$ on $\Omega$ such that
$$
|f|_{C^{r,\gamma}(\Omega)}:=\max_{|\alpha|=r}|\partial^\alpha\! f|_{\Omega,\gamma}<\infty,$$
where
$$
|f|_{C^{0,\gamma}(\Omega)}:=\sup_{\x,\y\in\Omega,\, \x\ne\y}\frac{|f(\x)-f(\y)|}{\|\x-\y\|^\gamma},\quad 0<\gamma\le 1.$$

Given $\z\in\RR^d$ and $\X=\{\x_1,\ldots,\x_N\}\subset\RR^d$, we denote by $B_{\z,\X}$ the closed ball in $\RR^d$ centered at the
origin and with radius equal to the diameter of the set $S_{\z,\X}$.

\begin{lemma}\label{holder}
Assume that $K(\x,\y)=\Phi(\x-\y)$, and $\Phi\in C^{r,\gamma}(B_{\z,\X})$. Then 
$K$ possesses continuous  partial derivatives  $\partial^{\alpha,\beta} K(\x,\y)$ for all
$|\alpha|,|\beta|\le m$, where $m=\lfloor r/2\rfloor$, and 
\begin{equation}\label{hest1}
|\partial^{\alpha,\beta}K|_{\z,\X,(r-2m+\gamma)/2}\le \max\{2,\sqrt{d}\}|\Phi|_{C^{r,\gamma}(B_{\z,\X})},
\quad |\alpha|=|\beta|=m,
\end{equation}
\begin{equation}\label{hest2}
|K|_{\z,\X,(r+\gamma)/2}\le C_{d,r}|\Phi|_{C^{r,\gamma}(B_{\z,\X})},
\end{equation}
where 
\begin{equation}\label{Cdr}
C_{d,r}=
\begin{cases}
\frac{2d^{r/2}}{(\gamma/2+1)^2\cdots (\gamma/2+\lfloor r/2\rfloor)^2} &\text{if $r$ is even},\\
\frac{d^{r/2}}{((1+\gamma)/2+1)^2\cdots ((1+\gamma)/2+\lfloor r/2\rfloor)^2} &\text{if $r$ is odd.}
\end{cases}
\end{equation}
\end{lemma}

\begin{proof}
The first statement follows because 
$$\partial^{\alpha,\beta} K(\x,\y)
=\frac{\partial^{|\alpha|}}{\partial \x^{\alpha}}\frac{\partial^{|\beta|}}{\partial \y^{\beta}}\Phi(\x-\y)
=(-1)^{|\beta|}\partial^{\alpha+\beta}\Phi(\x-\y).$$
To show \eqref{hest1} and \eqref{hest2}, we consider separately the cases of an even and an odd $r$.

Let $r=2m$. Then $(r+\gamma)/2=m+\gamma/2$. Let $|\alpha|=|\beta|=m$, hence $|\alpha+\beta|=r$. For 
$U(\x,\y)=\partial^{\alpha+\beta}\Phi(\x-\y)$, let
$$
\Delta_{\z,\x,\y}U=U(\x,\y)-U(\x,\z)-U(\z,\y)+U(\z,\z).$$
As soon as $\x,\y\in S_{\z,\X}$ we have
\begin{align*}
|\Delta_{\z,\x,\y}U|&\le\big|\partial^{\alpha+\beta}\Phi(\x-\y)-\partial^{\alpha+\beta}\Phi(\x-\z)\big|
+\big|\partial^{\alpha+\beta}\Phi(\z-\y)-\partial^{\alpha+\beta}\Phi(0)\big|\\
&\le 2|\partial^{\alpha+\beta}\Phi|_{C^{0,\gamma}(B_{\z,\X})}\,\|\y-\z\|_2^\gamma,\quad\text{and} \\
|\Delta_{\z,\x,\y}U|&\le\big|\partial^{\alpha+\beta}\Phi(\x-\y)-\partial^{\alpha+\beta}\Phi(\z-\y)\big|
+\big|\partial^{\alpha+\beta}\Phi(\x-\z)-\partial^{\alpha+\beta}\Phi(0)\big|\\
&\le 2|\partial^{\alpha+\beta}\Phi|_{C^{0,\gamma}(B_{\z,\X})}\,\|\x-\z\|_2^\gamma,
\end{align*}
since all points $\x-\y,\x-\z,\z-\y,0$ belong to $B_{\z,\X}$. It follows that
\begin{align*}
|\Delta_{\z,\x,\y}U|&\le 2|\partial^{\alpha+\beta}\Phi|_{C^{0,\gamma}(B_{\z,\X})}\min\{\|\x-\z\|_2^\gamma,\|\y-\z\|_2^\gamma\}\\
&\le 2|\partial^{\alpha+\beta}\Phi|_{C^{0,\gamma}(B_{\z,\X})}\,\|\x-\z\|_2^{\gamma/2}\|\y-\z\|_2^{\gamma/2},
\end{align*}
which implies \eqref{hest1}.

Let now $r=2m+1$ and $|\alpha|=|\beta|=m$. Then $|\alpha+\beta|=r-1$, and hence 
$U(\x,\y)=\partial^{\alpha+\beta}\Phi(\x-\y)=:\Phi_{\alpha,\beta}(\x-\y)$ is continuously differentiable, with its partial
derivatives of the first order in $C^{0,\gamma}(B_{\z,\X})$. Therefore,
with $\y^s:=\z+s(\y-\z)\in S_{\z,\X}$, we have
\begin{align*}
U(\x,\y)-U(\x,\z)&=-\int_0^1\nabla \Phi_{\alpha,\beta}(\x-\y^s)\cdot (\y-\z)\,ds,\\
U(\z,\y)-U(\z,\z)&=-\int_0^1\nabla \Phi_{\alpha,\beta}(\z-\y^s)\cdot (\y-\z)\,ds,
\end{align*}
such that
\begin{align*}
|\Delta_{\z,\x,\y}U|
&=\Big|\int_0^1\big(\nabla \Phi_{\alpha,\beta}(\x-\y^s)-\nabla \Phi_{\alpha,\beta}(\z-\y^s)\big)\cdot (\y-\z)\,ds\Big|\\
&\le \sqrt{d}|\Phi|_{C^{r,\gamma}(B_{\z,\X})}\,\|\x-\z\|_2^\gamma\|\y-\z\|.
\end{align*}
Similarly to the above, we obtain
$$
|\Delta_{\z,\x,\y}U|\le \sqrt{d}|\Phi|_{C^{r,\gamma}(B_{\z,\X})}\,\|\x-\z\|_2\|\y-\z\|^\gamma.$$
Hence
$$
|\Delta_{\z,\x,\y}U|\le \sqrt{d}|\Phi|_{C^{r,\gamma}(B_{\z,\X})}\,\|\x-\z\|_2^{(1+\gamma)/2}\|\y-\z\|^{(1+\gamma)/2},$$
which implies \eqref{hest1} also in this case.

The estimate \eqref{hest2} follows from the above using the definition of $|K|_{\z,\X,m+\theta}$, where we take 
$\theta=\gamma/2$ if $r$ is even, and $\theta=(1+\gamma)/2$ if $r$ is odd.
\end{proof}

Using this lemma we immediately obtain from Theorem~\ref{pestbt} the following error bound for translation-invariant kernels.

\begin{corollary}\label{pestPhi}
Let $D$ be a linear differential operator of order $k$, and $K(\x,\y)=\Phi(\x-\y)$ a translation-invariant
cpd kernel of order $s$ on $\RR^d$. 
Given $\z\in\RR^d$ and $\X=\{\x_1,\ldots,\x_N\}\subset\RR^d$,  assume that  the linear system 
\eqref{wsysp} is consistent, and $\Phi\in C^{r,\gamma}(B_{\z,\X})$ for some $r\ge 2k$ and $\gamma\in(0,1]$.
Then for all $f\in {\cal F}_{K}$ the kernel-based numerical differentiation formula
\eqref{ndif} satisfies the error bound
\begin{equation}\label{gestPhi}
\Big|Df(\z) -\sum_{j=1}^N w^*_jf(\x_j)\Big|\le \rho_{q,D}(\z,\X,\tfrac{r+\gamma}{2})
\,C_{d,r}^{1/2}|\Phi|_{C^{r,\gamma}(B_{\z,\X})}^{1/2}\,\|f\|_{K},
\end{equation}
where $q=\max\{s, \lfloor r/2\rfloor+1\}$, and $C_{d,r}$ is given by \eqref{Cdr}.
\end{corollary}

\begin{remark}
Clearly, using either \cite[Theorem 9]{DS16} or Corollary~\ref{estold} we could only obtain a bound in terms of the growth
function $\rho_{q,D}(\z,\X,q)$ with $q\le m=\lfloor r/2\rfloor$, and hence at best an 
$\cO(h_{\z,\X}^{\lfloor r/2\rfloor-k})$ estimate of the numerical differentiation error, rather than 
$\cO(h_{\z,\X}^{(r+\gamma)/2-k})$ by Corollary~\ref{pestPhi}.
\end{remark}

\medskip

We now apply Corollary~\ref{pestPhi} to several families of radial basis functions with finite smoothness. Their properties
can be found for example in the books \cite{Buhmann03,Fasshauer07,Wendland}.

\medskip

\noindent
\textbf{Polyharmonic spline:}\quad $\Phi(\x)=\Phi_{\text{PHS},\nu}(\x):=(-1)^{\lceil\nu/2\rceil}\|\x\|_2^\nu$, $\nu>0$,
$\nu\notin 2\NN$, conditionally positive definite of order $s\ge\lceil\nu/2\rceil$. It can be easily shown that 
$\Phi_{\text{PHS},\nu}\in C^{r,\gamma}(B)$ for any ball $B\subset\RR^d$ centered at the origin, where 
$r=\lceil\nu\rceil-1$, $\gamma=\nu-r$.

\begin{corollary}\label{pestPHS}
Let $D$ be a linear differential operator of order $k$, and $K(\x,\y)=\Phi_{\text{PHS},\nu}(\x-\y)$ for some  
$\nu>2k$, $\nu\notin 2\NN$, and let $s\ge\lceil\nu/2\rceil$.
Given $\z\in\RR^d$ and $\X=\{\x_1,\ldots,\x_N\}\subset\RR^d$,  assume that  the linear system 
\eqref{wsysp} is consistent.
Then for all $f\in {\cal F}_{K}$ the kernel-based numerical differentiation formula
\eqref{ndif} satisfies the error bound
$$%
\Big|Df(\z) -\sum_{j=1}^N w^*_jf(\x_j)\Big|\le \rho_{s,D}(\z,\X,\tfrac{\nu}{2})
\,C_{d,r}^{1/2}|\Phi_{\text{PHS},\nu}|_{C^{r,\gamma}(B_{\z,\X})}^{1/2}\,\|f\|_{K},
$$%
where 
$r=\lceil\nu\rceil-1$, $\gamma=\nu-r$,
and $C_{d,r}$ is given by \eqref{Cdr}.
\end{corollary}

\medskip

\noindent
\textbf{Thin plate splines:}\quad $\Phi(\x)=\Phi_{\text{TPS},n}(\x):=(-1)^{n+1}\|\x\|_2^{2n}\log \|\x\|_2$, $n\in \NN$, 
conditionally positive definite of order $s\ge n+1$. As shown for example in \cite[p.~186]{Wendland},  
$\Phi_{\text{TPS},n}\in C^{2n-1,\gamma}(B)$ for any ball $B\subset\RR^d$ centered at the origin and every $0<\gamma<1$.

\begin{corollary}\label{pestTPS}
Let $D$ be a linear differential operator of order $k$, and $K(\x,\y)=\Phi_{\text{TPS},n}(\x-\y)$ for some  
$n\ge k+1$, and let $s\ge n+1$.
Given $\z\in\RR^d$ and $\X=\{\x_1,\ldots,\x_N\}\subset\RR^d$,  assume that  the linear system 
\eqref{wsysp} is consistent.
Then for all $f\in {\cal F}_{K}$ the kernel-based numerical differentiation formula
\eqref{ndif} satisfies for any $0<\gamma<1$ the error bound
$$%
\Big|Df(\z) -\sum_{j=1}^N w^*_jf(\x_j)\Big|\le \rho_{s,D}(\z,\X,n-\tfrac{1-\gamma}{2})
\,C_{d,2n-1}^{1/2}|\Phi_{\text{TPS},n}|_{C^{2n-1,\gamma}(B_{\z,\X})}^{1/2}\,\|f\|_{K},
$$%
where $C_{d,2n-1}$ is given by \eqref{Cdr}.
\end{corollary}

\medskip

\noindent
\textbf{Wendland's compactly supported radial basis functions:}\quad 
$\Phi(\x)=\Phi_{d,n}(\x)=\varphi_{d,n}(\|\x\|_2)$, $n=0,1,\ldots$, see \cite[Section 9.4]{Wendland} for a definition. They are strictly 
positive definite, and hence we can use any $s\ge0$. It is shown in  \cite{Wendland} that $\Phi_{d,n}$ is $2n$ times
continuously differentiable in $\RR^d$. Since $\varphi_{d,n}$ are piecewise polynomials, it follows that 
$\Phi_{d,n}\in C^{2n,1}(B)$ for any ball $B\subset\RR^d$ centered at the origin.

\begin{corollary}\label{pestWend}
Let $D$ be a linear differential operator of order $k$, and $K(\x,\y)=\Phi_{d,n}(\x-\y)$ for some  
$n\ge k$, and let $s\ge 0$.
Let $\z\in\RR^d$ and $\X=\{\x_1,\ldots,\x_N\}\subset\RR^d$. 
If $s>0$, we assume  that  the linear system  \eqref{wsysp} is consistent.
Then for all $f\in {\cal F}_{K}$ the kernel-based numerical differentiation formula
\eqref{ndif} satisfies  the error bound
$$
\Big|Df(\z) -\sum_{j=1}^N w^*_jf(\x_j)\Big|\le \rho_{q,D}(\z,\X,n+\tfrac12)
\,C_{d,2n}^{1/2}|\Phi_{d,n}|_{C^{2n,1}(B_{\z,\X})}^{1/2}\,\|f\|_{K},
$$
where $q=\max\{s, n+1\}$, and $C_{d,2n}$ is given by \eqref{Cdr}.
\end{corollary}

\bibliographystyle{abbrv}
\bibliography{meshless}

\begin{thebibliography}{1}

\bibitem{BFFB17}
V.~Bayona, N.~Flyer, B.~Fornberg, and G.~A. Barnett.
\newblock On the role of polynomials in {RBF-FD} approximations: {II}.
  {N}umerical solution of elliptic {PDE}s.
\newblock {\em Journal of Computational Physics}, 332:257 -- 273, 2017.

\bibitem{Buhmann03}
M.~D. Buhmann.
\newblock {\em Radial Basis Functions}.
\newblock Cambridge University Press, New York, NY, USA, 2003.

\bibitem{D21}
O.~Davydov.
\newblock Approximation with conditionally positive definite kernels on
  deficient sets.
\newblock In M.~N. Gregory E.~Fasshauer and L.~L. Schumaker, editors, {\em
  Approximation Theory XVI: Nashville 2019}, pages 27--38. Springer Berlin
  Heidelberg, 2021.

\bibitem{DS16}
O.~Davydov and R.~Schaback.
\newblock Error bounds for kernel-based numerical differentiation.
\newblock {\em Numerische Mathematik}, 132(2):243--269, 2016.

\bibitem{DS18}
O.~Davydov and R.~Schaback.
\newblock Minimal numerical differentiation formulas.
\newblock {\em Numerische Mathematik}, 140(3):555--592, 2018.

\bibitem{Fasshauer07}
G.~Fasshauer.
\newblock {\em Meshfree Approximation Methods with MATLAB}.
\newblock World Scientific Publishing Co., Inc., River Edge, NJ, USA, 2007.

\bibitem{FFprimer15}
B.~Fornberg and N.~Flyer.
\newblock {\em A Primer on Radial Basis Functions with Applications to the
  Geosciences}.
\newblock Society for Industrial and Applied Mathematics, Philadelphia, PA,
  USA, 2015.

\bibitem{Schaback99}
R.~Schaback.
\newblock Native {H}ilbert spaces for radial basis functions {I}.
\newblock In M.~W. M{\"u}ller, M.~D. Buhmann, D.~H. Mache, and M.~Felten,
  editors, {\em New Developments in Approximation Theory: 2nd International
  Dortmund Meeting (IDoMAT), Germany, 1998}, pages 255--282, Basel, 1999.
  Birkh{\"a}user Basel.

\bibitem{Wendland}
H.~Wendland.
\newblock {\em Scattered Data Approximation}.
\newblock Cambridge University Press, 2004.

\end{thebibliography}

\end{document}